\documentclass{article}
\usepackage[T1]{fontenc}
\usepackage[latin9]{inputenc}
\usepackage{amssymb}
\usepackage{amsmath}
\usepackage{graphicx}
\usepackage{enumerate}
 
\usepackage[english]{babel}

\newcommand{\N}{\mathbb{N}}
\newcommand{\Z}{\mathbb{Z}}

\newcommand{\cH}{\mathcal{H}}
\newcommand{\cU}{\mathcal{U}}
\newcommand{\Ind}{\mathrm{Ind}}
\newcommand{\ind}{\mathrm{ind}}
\newcommand{\M}{\mathrm{M}}

\newtheorem{thm}{Theorem}
\newtheorem{prop}[thm]{Proposition}

\newtheorem{lem}[thm]{Lemma}
\newtheorem{fact}{Fact}
\newtheorem{facts}[fact]{Facts}
\newtheorem{problem}{Problem}

\newenvironment{proof}{{\bf Proof:}}{$\Box$\newline}

\newif\ifskip
\skiptrue

\begin{document}
\title{On $P$-unique hypergraphs}

\author{J.A. Makowsky  \\
Department of Computer Science \\
Technion--Israel Institute of Technology\\
Haifa, Israel\\
and
\\
R.X. Zhang \\
Department of Mathematics and Mathematics Education \\
National Institute of Education \\
Nanyang Technological University\\
Singapore
}   

\maketitle
\begin{center}
\abstract{We study hypergraphs which are uniquely determined by their 
chromatic,
independence  and
matching
polynomials.
B. Bollob\'as, L. Pebody and O. Riordan (2000) conjectured 
(BPR-conjecture) that almost all graphs
are uniquely determined by their chromatic polynomials.
We show that for $r$-uniform hypergraphs with $r \geq 3$ this is almost never the case.
This disproves the analolgue of the BPR-conjecture for $3$-uniform hypergraphs.
For $r =2$ this also holds for the independence polynomial,
as shown by J.A. Makowsky and V. Rakita (2017),
whereas for the chromatic and matching polynomial this remains open.
}
\end{center}

\section{Introduction and outline}

\subsection{Background}
A hypergraph  $H$ consists of a set $V(H)$ together with a family of subsets $E(H)$ of $V(H)$
called hyperedges. Two vertices $u,v \in V(H)$ are adjacent if there is a hyperedge $e \in E(H)$
such that both $u \in e$ and $v \in e$.
$H$ is $r$-uniform if every hyperedge in $E(H)$ has exactly $r$ elements.
Two hypergraphs $H_1, H_2$ are isomorphic,
denoted by $H_1 \simeq H_2$,
if there is a bijective map $h: V(H_1) \rightarrow V(H_2)$
such that for any two vertices $u, v \in V(H_1)$  we have $u$ and $v$ are adjacent iff
$h(u)$ and $h(v)$ are adjacent in $H_2$.

We denote by 
$\frak{H}$ ($\frak{H}^r$) 
the class of all ($r$-uniform) hypergraphs,
and by
$\frak{H}_n$ ($\frak{H}_n^r$) 
the set of all ($r$-uniform) hypergraphs $H$ with $V(H) = [n]$.
Here $[n] = \{ 1, 2, \ldots, n\}$.
We note that graphs are $2$-uniform hypergraphs.

A (univariate) hypergraph polynomial $P(H;X)$ is a function
$ P: \frak{H} \rightarrow \Z[X]$ 
which preserves hypergraph isomorphisms.
Let $P(H;X)$ be a univariate hypergraph polynomial.
A hypergraph $H$ is $P$-unique if for every hypergraph $H_1$ with $P(H_1:X) = P(H;X)$
we have that $H_1$ is isomorphic to $H$. 
Similarly,
a $r$-uniform hypergraph $H$ is $r$-$P$-unique if for every  $r$-uniform hypergraph $H_1$ with $P(H_1:X) = P(H;X)$
we have that $H_1$ is isomorphic to $H$. 

A hypergraph polynomial $P(H;X)$ is {\em complete 
(for $r$-uniform hypergraphs)
}, if all 
($r$-uniform) hypergraphs are $P$-unique.
Let $\cH(n)$ be the number of non-isomorphic hypergraphs on $n$ vertices, and let
and $\cH^r(n)$ be the number of non-isomorphic $r$-uniform hypergraphs on $n$ vertices. 
Furthermore, let
$\cU_P(n)$ 
($\cU_P^r(n)$) 
be the number of non-isomorphic $P$-unique ($r$-uniform)  hypergraphs on $n$ vertices. 

$P$ is {\em almost complete} 
if 
$$
\lim_{n \rightarrow \infty} \frac{\cU_P(n)}{\cH(n)} = 1.
$$
It is almost complete on
on $r$-uniform hypergraphs  if
$$
\lim_{n \rightarrow \infty} \frac{\cU_P^r(n)}{\cH^r(n)} = 1.
$$
$P$ is {\em weakly distinguishing} if 
$$
\lim_{n \rightarrow \infty} \frac{\cU_P(n)}{\cH(n)} = 0.
$$
and analogously for $r$-uniform hypergraphs.

No known 
univariate hypergraph polynomial
in the literature 
is complete, although one can construct such polynomials using some clever encoding of the isomorphism types
of finite hypergraphs.

\subsection{Three hypergraph polynomials}
Hypergraph polynomials studied in the literature are the chromatic polynomial $\chi(H,X)$
\cite{ar:Tomescu-2004,ar:Tomescu-2007,ar:Tomescu-2014,ar:BorowieckiLazuka-2000,ar:BorowieckiLazuka-2007},
the independence polynomial $\Ind(H;X))$,
\cite{ar:Trinks2016},
and the matching polynomial $\M(H;X)$,
\cite{ar:GuoZhaoMao2017}.
There where also attempts to extend spectral graph theory to hypergraphs, cf.
\cite{ar:CooperDutle-2012,ar:PearsonZhang-2014} and the references therein.
The monograph \cite{bk:BrouwerHaemers2012} 
summarizes what is known about graphs unique
for the characteristic and the Laplacian polynomial.
In \cite{bk:BrouwerHaemers2012}  the authors also
suggest that the characteristic polynomial
is almost complete on graphs. 
However, we do not discuss here which hypergraphs are unique for these polynomials.

\paragraph{The chromatic polynomial.}
The chromatic polynomial for hypergraphs defined below generalizes the chromatic polynomial for graphs,
but also show distinctly different behaviour in the case of hypergraphs, cf. \cite{ar:ZhangDong-2017}.

Let $k \in \N$ and
$f: V(H) \rightarrow [k]$.
$f$ is a {\em proper coloring of $H$ with at most $k$ colors}
if every $e \in E(H)$ contains two vertices $u,v \in e$ with $f(u) \neq f(v)$.
We denote by $\chi(H;k)$ the number of proper colorings of $H$ with at most $k$ colors.

A set $I \subseteq V(H)$ is {\em independent} of there is no edge $e \subseteq I$.
For $i \in \N$
let $b_i(H)$ be number of partitions of $V(H)$ into $i$ independent sets. 
For $X$ we denote by $X_{(i)}$ the polynomial
$$
X \cdot (X_1) \cdot \ldots \cdot (X-i+1).
$$

\begin{prop}[\cite{ar:BorowieckiLazuka-2000,Zhang}]
\label{pr:chromatic}
\begin{enumerate}[(i)]
\item
$\chi(H;k)$ is a polynomial in $k$, hence can be extended to a polynomial $\chi(H;X) \in \Z[X]$.
\item
$\chi(H;X) = \sum_i^n b_i(H) \cdot X_{(i)}$.
\end{enumerate}
\end{prop}

$\chi$-unique hypergraphs were presented compactly  in \cite{Zhang}
which we summarize here. 
The definitions of hypercycles, hyperpaths and sunflower hypergraphs
are standard in the hypergraph literature, cf.
the books by 
C. Berge \cite{Berge73}.
V. I. Voloshin \cite{bk:Voloshin-2009}, 
and 
A. Bretto \cite{bk:Bretto-2013}.

\begin{prop}[\cite{ar:Tomescu-2004,ar:Tomescu-2007,ar:Tomescu-2014,Zhang}]
\begin{enumerate}[(i)]
\item
For $r\geq 3$, $r$-uniform hypercycle $\mathcal{C}_{m}^{r}$ is $r$-$\chi$-unique but it is not $\chi$-unique;
\item
For every $p, r\geq 3$, $B_{p, p}^{2, 2}$ is $\chi$-unique, where $B_{p,p}^{2, 2}$ 
denote the hypergraphs obtained by identifying extremities $x$ and $y$ of 
$\mathcal{P}_{p-1}^{r}$ with distinct vertices of degree $2$ in the same edge of $\mathcal{C}_{p}^{r}$;
\item
The sunflower hypergraph $SH(n, 1, r)$ is $\chi$-unique; 
\item
Let $n=r+(k-1)p$, where $r\geq 3, k\geq 1$ and $1\leq p\leq r-1$. 
Then $SH(n, p, r)$ is $r$-$\chi$-unique for every $1\leq p\leq r-2$; for $p=r-1$, 
$SH(n, r-1, r)$ is $r$-$\chi$-unique for $k=1$ or $k=2$ but it has not this property for $k\geq 3$.
\end{enumerate}
\end{prop}

In \cite{ar:BollobasPebodyRiordan2000} it is conjectured that the chromatic polynomial
is almost complete on graphs.
Our first result shows that the conjecture is not true for hypergraphs.
\begin{thm}
\label{th:main-chrom}
The chromatic polynomial
$\chi(H;X)$ is weakly distinguishing
\begin{enumerate}[(i)]
\item
on hypergraphs:
$$\lim_{n \rightarrow \infty} \frac{\cU_{\chi}(n)}{\cH(n)} = 0.$$
\\
\item
on $r$-uniform hypergraphs:
For every $r \geq 3$
$$\lim_{n \rightarrow \infty} \frac{\cU_{\chi}^r(n)}{\cH^r(n)} = 0.$$
\end{enumerate}
\end{thm}

\paragraph{The independence polynomial.}
The independence polynomial for hypergraphs is defined as
$$
\Ind(H:X) =
\sum_i \ind_i(H) \cdot X^i.
$$
where  $\ind_i(H)$ is the number of independent sets $I \subseteq V(H)$ with $|I|=i$.

The independence polynomial for hypergraphs was studied in
\cite{ar:Trinks2016}.

\begin{thm}
\label{th:main-indep}
The independence polynomial $\Ind(H;X)$ is weakly distinguishing
\begin{enumerate}[(i)]
\item
on hypergraphs:
$$ 
\lim_{n \rightarrow \infty} \frac{\cU_{Ind}(n)}{\cH(n)} = 0. 
$$
\item
on $r$-uniform hypergraphs:
For every $r \geq 2$
$$ 
\lim_{n \rightarrow \infty} \frac{\cU_{Ind}^r(n)}{\cH^r(n)} = 0. 
$$
\end{enumerate}
\end{thm}

The case $r=2$ was shown in \cite{ar:JAMRakita2017}.
The proof for $r \geq 3$ is given in Section \ref{se:indep}.

\paragraph{The matching polynomial.}
A {\em $k$-matching $m$ of a hypergraph $H$} is a set $m \subseteq E(H)$ of $k$ disjoint
hyperedges.
Let $\mu_k(H)$ be the number of $k$-matchings of $H$.
The {\em matching polynomial $\M(H;X)$} of a hypergraph $H$ is defined by
$$
\M(H;X) = \sum_{k=1}^{\lfloor\frac{n}{k_H}\rfloor}\mu_k(H) \cdot X^k.
$$
where $k_H$ is the minimum size of the edges in $E(H)$.
The matching polynomial for hypergraphs was studied \cite{ar:GuoZhaoMao2017}.
M. Noy  \cite{ar:Noy03} studied $\M$-unique graphs.

\begin{thm}
\label{th:main-match}
The matching polynomial
$\M(H;X)$ is weakly distinguishing
\begin{enumerate}[(i)]
\item
on hypergraphs:
$$ \lim_{n \rightarrow \infty} \frac{\cU_{M}(n)}{\cH(n)} = 0.  $$
\item
on $r$-uniform hypergraphs:
For every $r \geq 3$
$$ \lim_{n \rightarrow \infty} \frac{\cU_{M}^r(n)}{\cH^r(n)} = 0. 
$$
\end{enumerate}
\end{thm}
The case $r=2$ is still open.
The proof is also given in Section \ref{se:indep}.

To the best of our knowledge
no explicite description of
$\Ind$-unique 
and
$\M$-unique 
hypergraphs is given in the literature.

\section{General strategy of the proofs}
\label{se:strategy}

The general strategy of our proofs has been motivated by the first author's work with V. Rakita
\cite{ar:JAMRakita2017}.

Let $\overline{\cH}(n)$ and $\overline{\cH}^r(n)$ denote the number of {\em labeled} hypergraphs and
$r$-uniform hypergraphs of order $n$.

\begin{lem}
\label{le:labeled}
For any positive integer $n$, 
\begin{enumerate}[(i)]
\item
$\overline{\cH}(n) \leq 2^{2^n}$.
\item
$ \overline{\mathcal{H}}(n) \leq \mathcal{H}(n) \cdot n!$
\item
$\overline{\cH}^r(n) \leq 2^{{n \choose r}}$.
\item
$ \overline{\mathcal{H}}^r(n) \leq \mathcal{H}^r(n) \cdot n!$
\end{enumerate}
\end{lem}
\begin{proof}
(i) and (iii):
are obvious.
\\
(ii) and (iv): 
Not every ($r$-uniform) hypergraph has automorphisms.
\end{proof}

Let $P(H;X)$ be a hypergraph polynomial.
We denote by
$\mathcal{B}_{P}(n)$ 
($\mathcal{B}_{P}^r(n)$) 
the number of polynomials $p(X)$ such that there is a  ($r$-uniform) hypergraph $H$
of order $n$ with $p(X)=P(H;X)$.

We shall use the following observations:

\begin{lem}
\label{le:strategy}
\begin{enumerate}[(i)]
\item
$\mathcal{B}_{P}^r(n) \leq \mathcal{B}_{P}(n)$.
\item
$U_{P}(n)\leq \mathcal{B}_{P}(n)\leq \mathcal{H}(n)$. 
\item
$U_{P}^r(n)\leq \mathcal{B}_{P}^r(n)\leq \mathcal{H}^r(n)$. 
\end{enumerate}
\end{lem}
\begin{proof}
(i): There are more candidate polynomials for
$\mathcal{B}_{P}(n)$ than  for
$\mathcal{B}_{P}^r(n)$.
\\
(ii) and (iii):
There cannot be more $P$-unique hypergraphs than polynomials in $\mathcal{B}_{P}(n)$ ($\mathcal{B}_{P}^r(n)$).
There cannot be more 
polynomials in $\mathcal{B}_{P}(n)$ ($\mathcal{B}_{P}^r(n)$)
than there are hypergraphs in  $\mathcal{H}(n)$ ($\mathcal{H}^r(n)$).
\end{proof}

To prove our Theorems we will use Lemmas  \ref{le:labeled} and \ref{le:strategy} and estimate the numbers
$\mathcal{B}_{P}^r(n)$ or
$\mathcal{B}_{P}(n)$.

We also use an observation from pre-calculus:

\begin{prop}
\label{pr:precalculus}
Let $f, g$ be two non-decreasing real functions.
\\
If $\lim_{n \rightarrow \infty} \frac{\log(f(n))}{\log(g(n))} =0$ then
$\lim_{n \rightarrow \infty} \frac{f(n)}{g(n)} =0$.
\end{prop}

\section{Proof of Theorem \ref{th:main-chrom}}
\label{se:chromatic}
By Proposition \ref{pr:chromatic}(ii)
the chromatic polynomial of hypergraphs can be written as 
$$\chi(H;X)=\sum_{i=1}^{n}b_i(H) \cdot X_{(i)}.$$

where $b_i(H)$
is the number of partitions of $V(H)$
into $i$ non-empty independent subsets.
Let
$S(n,i)$ denote the Stirling number of the second kind, which counts the number of partitions of $[n]$
into $i$ distinct subsets.

\begin{lem}
$\mathcal{B}_{\chi}(n)\leq \prod_{i=1}^{n}S(n,i).$
\end{lem}
\begin{proof}
Clearly, $0\leq  b_i(H) \leq S(n, i)$. 
\end{proof}

\begin{thm}
\label{th:1}
$$\lim_{n\rightarrow \infty}\frac{\mathcal{B}_{\chi}(n)}{\mathcal{H}(n)}=0.$$
\end{thm}

We first recall some facts from \cite{bk:GrahamKnuthPatashnik94}.
\begin{facts}
\begin{enumerate}[(i)]
\item
The Stirling approximation for the factorial:
\begin{gather}
n!\sim \sqrt{2\pi n}(\frac{n}{e})^{n}.
\notag
\end{gather}
\item
For fixed $n$, $S(n,i)$ is unimodal in $i$, i.e., 
it has a single maximum, which is attained for at most two consecutive values of $i$. 
That is, there is an integer $K_{n}$ such that 
\begin{gather}
S(n, 1)<S(n, 2)<\cdots<S(n, K_{n}) \notag \\
S(n, K_{n})\geq S(n, K_{n}+1)>\cdots>S(n, n).
\notag
\end{gather}
\item
When $n$ is large enough, $K_{n}\sim \frac{n}{\log n}$. 
\end{enumerate}
\end{facts}

Now we  prove Theorem \ref{th:1}.

\begin{proof}
\begin{gather}
\frac{\mathcal{B}_{\chi}(n)}{\mathcal{H}(n)} \leq
\frac{\mathcal{B}_{\chi}(n)\cdot n!}{\overline{\mathcal{H}}(n)}= 
\frac{\mathcal{B}_{\chi}(n)\cdot n!}{2^{2^{n}}} 
\notag 
\end{gather}

\begin{gather}
\frac{\mathcal{B}_{\chi}(n)\cdot n!}{2^{2^{n}}} 
\leq  \frac{\prod_{i=1}^{n}S(n, i)\cdot n!}{2^{2^{n}}}
\leq \frac{(S(n, K_{n}))^{n}\cdot n!}{2^{2^{n}}}
\notag \\
\sim \frac{(S(n, \frac{n}{\log n}))^{n}\cdot n!}{2^{2^{n}}}.
\notag
\end{gather}
Let us define 
$\mathcal{B}_{\chi}^{*}(n)=S(n,\frac{n}{\log n}).$

We estimate 
$\mathcal{B}_{\chi}^{*}(n)$.
\begin{gather}
\mathcal{B}_{\chi}^{*}(n) 
\sim 
\frac{(\frac{n}{\log n})^{n}}{(\frac{n}{log n})!}
=
\frac{n^{n}}{(\log n)^{n}\cdot (\frac{n}{\log n})!}
\notag \\
\sim \frac{n^{n}\cdot (e\cdot \log n)^{\frac{n}{\log n}}}{(\log n)^{n}
\cdot \sqrt{2\pi\cdot \frac{n}{\log n}}\cdot n^{\frac{n}{\log n}}}
\notag \\
=
\frac{n^{n-\frac{n}{\log n}}\cdot e^{\frac{n}{\log n}}\cdot (\log n)^{\frac{n}{\log n}-n}}{\sqrt{2\pi\cdot \frac{n}{\log n}}}
\leq n^{n}\cdot e^{n}. 
\notag
\end{gather}

Therefore, 
\begin{gather}
\frac{\mathcal{B}_{\chi}(n)}{\mathcal{H}(n)}\leq  \frac{(n^{n}\cdot e^{n})^{n}\cdot n!}{2^{2^{n}}}
\leq  \frac{n^{n^{2}}\cdot e^{n^{2}}\cdot n^{n}}{2^{2^{n}}}=\frac{n^{n^{2}+n}\cdot e^{n^{2}}}{2^{2^{n}}}.
\notag
\end{gather}

Now we use Proposition \ref{pr:precalculus} and
take 
base $2$ 
logarithms 
of the numerator and denominator, and we get:
\begin{gather}
\label{eq:endofproof-chi}
\lim_{n\rightarrow \infty}\frac{\log_{2} \mathcal{B}_{\chi}(n)}{\log_{2} \mathcal{H}(n)}
\leq 
\lim_{n\rightarrow \infty}\frac{(n^{2}+n)\log_{2}n+n^{2}\log_{2}e}{2^{n}}=0,
\end{gather}
which implies that $\lim_{n\rightarrow \infty}\frac{\mathcal{B}_{\chi}(n)}{\mathcal{H}(n)}=0$.
\end{proof}

Now Theorem \ref{th:main-chrom}(i) follows using Lemma \ref{le:strategy}(ii).
Theorem \ref{th:main-chrom}(ii) follows 
using Lemma \ref{le:strategy}(i) and (iii)
by replacing
$\cH(n)$ by $\cH^r(n) \leq 2^{{n \choose r}}$ in
the proof of Theorem  \ref{th:1}.
We get instead of Equation (\ref{eq:endofproof-chi})
\begin{gather}
\label{eq:r-uniform}
\lim_{n\rightarrow \infty}\frac{\log_{2} \mathcal{B}_{\chi}^r(n)}{\log_{2} \mathcal{H}^r(n)}
\leq 
\lim_{n\rightarrow \infty}\frac{(n^{2}+n)\log_{2}n+n^{2}\log_{2}e}{{n \choose r}}=0,
\notag
\end{gather}
which still holds for $r \geq 3$.



\section{Proof of Theorems \ref{th:main-indep} and \ref{th:main-match}}
\label{se:indep}
Let $\mathcal{B}_{\Ind}(n)$ be the number of polynomials $p(X)$ such that there is a 
hypergraph $H$ of order $n$ with $p(X)=\Ind(H;X)$. 

\begin{thm}
\label{th:2}
\begin{enumerate}[(i)]
\item
$
\lim_{n\rightarrow \infty}
\frac{{\mathcal{B}_{\Ind}(n)}}
{\cH^(n)}=0.
$
\item
$
\lim_{n\rightarrow \infty}
\frac{{\mathcal{B}^{r}_{\Ind}(n)}}
{\cH^{r}(n)}=0
$.
\end{enumerate}
\end{thm}
\begin{proof}
(i):
As $0\leq \ind_{i}(H)\leq \binom{n}{i}$, then 
$$
\mathcal{B}_{\Ind}(n)
\leq
\prod_{i=1}^{n}\binom{n}{i}
<
\prod_{i=1}^{n}(\frac{n\cdot e}{i})^{i}
<
\prod_{i=1}^{n}(n\cdot e)^{i}.
$$

From Lemma \ref{le:labeled}, we have that 
\begin{gather}
\label{eq 1}
\frac{\mathcal{B}_{\Ind}(n)}{\mathcal{H}(n)}
\leq 
\frac{\mathcal{B}_{\Ind}(n)\cdot n!}{\overline{\mathcal{H}}(n)}
<
\frac{\prod_{i=1}^{n}(n\cdot e)^{i}\cdot n!}{2^{2^{n}}}
<
\frac{\prod_{i=1}^{n}(n\cdot e)^{i}\cdot n^{n}}{2^{2^{n}}}.
\end{gather}

Now we use Proposition \ref{pr:precalculus} and
take 
base $2$ 
logarithms 
of the numerator and denominator, and we get in
Equation (\ref{eq 1}).

\begin{gather}
\frac{\log\mathcal{B}_{\Ind}(n)}{\log \overline{H}(n)}
<
\frac{\frac{n\cdot (n+1)}{2}\cdot \log(n\cdot e)+n\cdot \log n}{2^{n}\cdot log 2}.
\notag
\end{gather}
Therefore, 
\begin{gather}
\label{eq 2}
\lim_{n\rightarrow\infty}\frac{\log\mathcal{B}_{\Ind}(n)}{\log \overline{H}(n)}
<
\lim_{n\rightarrow\infty}\frac{\frac{n\cdot (n+1)}{2}\cdot \log(n\cdot e)+n\cdot \log n}{2^{n}\cdot log 2}
=0,
\end{gather}
which implies that 
$$
\lim_{n\rightarrow \infty}\frac{\mathcal{B}_{\Ind}(n)}{\cH(n)}=0.
$$
\\
(ii):
As in the case of the chromatic polynomial
we replace $\cH(n)$ by $\cH^r(n) \leq 2^{{n \choose r}}$
and
$\mathcal{B}_{\Ind}(n)$ 
by
$\mathcal{B}_{\Ind}^r(n)$ 
in
Equation (\ref{eq 2}) and use
Lemma \ref{le:strategy}(i) and (iii).
\hfill
\end{proof}



Now Theorem \ref{th:main-indep}(i) and (ii) follow.


Let $\mathcal{B}_{\M}(n)$ be the number of polynomials $p(X)$ such that there is a 
hypergraph $H$ of order $n$ with $p(X)=\M(H;X)$.


\begin{thm}
\begin{enumerate}[(i)]
\item
$
\lim_{n\rightarrow \infty}\frac{\mathcal{B}_{\M}(n)}{\mathcal{H}(n)}=0.
$
\item
$
\lim_{n\rightarrow \infty}\frac{\mathcal{B}^{r}_{\M}(n)}{\mathcal{H}^{r}(n)}=0.
$
\end{enumerate}
\end{thm}
\begin{proof}
(i):
As $0\leq \omega_{i}(\mathcal{H})\leq \binom{\lfloor\frac{n}{k}\rfloor}{i}$, then $$\mathcal{B}_{\M}(n)
\leq
\prod_{i=1}^{n}\binom{\lfloor\frac{n}{k}\rfloor}{i}
<
\prod_{i=1}^{n}(\frac{\lfloor\frac{n}{k}\rfloor\cdot e}{i})^{i}
<
\prod_{i=1}^{n}(\lfloor\frac{n}{k}\rfloor\cdot e)^{i}.$$

From Lemma \ref{le:labeled}, we have that 
\begin{gather}\label{eq 4}
\frac{\mathcal{B}_{\M}(n)}{\mathcal{H}(n)}
\leq 
\frac{\mathcal{B}_{\M}(n)\cdot n!}{\overline{\mathcal{H}}(n)}
<
\frac{\prod_{i=1}^{n}(\lfloor\frac{n}{k}\rfloor\cdot e)^{i}\cdot n!}{2^{2^{n}}}
<
\frac{\prod_{i=1}^{n}(\lfloor\frac{n}{k}\rfloor\cdot e)^{i}\cdot n^{n}}{2^{2^{n}}}.
\end{gather}

Now we use Proposition \ref{pr:precalculus} and
take 
base $2$ 
logarithms 
of the numerator and denominator, and we get in
Equation (\ref{eq 4}):
\begin{gather}
\frac{\log\mathcal{B}_{\M}(n)}{\log \overline{\mathcal{H}}(n)}
<
\frac{\frac{\lfloor\frac{n}{k}\rfloor\cdot (\lfloor\frac{n}{k}\rfloor+1)}{2}\cdot \log(n\cdot e)+n\cdot \log n}{2^{n}\cdot log 2}.
\notag
\end{gather}
Therefore, 
\begin{gather}
\label{eq 5}
\lim_{n\rightarrow\infty}\frac{\log\mathcal{B}_{\M}(n)}{\log \overline{\mathcal{H}}(n)}
<
\lim_{n\rightarrow\infty}\frac{\frac{\lfloor\frac{n}{k}\rfloor\cdot (\lfloor\frac{n}{k}\rfloor+1)}{2}\cdot \log(n\cdot e)+n\cdot \log n}{2^{n}\cdot log 2}
=0,
\end{gather}
which implies that $\lim_{n\rightarrow \infty}\frac{\mathcal{B}_{\M}(n)}{\mathcal{H}(n)}=0$.
\\
(ii):
As in the case of the chromatic polynomial
we replace $\cH(n)$ by $\cH^r(n) \leq 2^{{n \choose r}}$
and
$\mathcal{B}_{\M}(n)$ 
by
$\mathcal{B}_{\M}^r(n)$ 
in
Equation (\ref{eq 5}) and use
Lemma \ref{le:strategy}(i) and (iii).
\hfill
\end{proof}

Now Theorem \ref{th:main-match}(i) and (ii) follow.

\section{Conclusions and further research}

We have shown that for $r$-uniform hypergraphs with $r \geq 3$, and hypergraphs in general,
there are very few hypergraphs which are unique for $\chi$, $\Ind$ and $\M$.
This is not so surprising as there are many more $r$-uniform hypergraphs of order $n$ than graphs.
Still, it is interesting to search for such graphs.

\begin{problem}
Find more $P$-unique $r$-uniform hypergraphs for $\chi$, $\Ind$ and $\M$.
\end{problem}
For $\Ind$ and $\M$ it seems this has not been properly investigated.

In \cite{ar:White2011} the Tutte polynomial $T(G;X,Y)$ and the most general edge elimnation polynomial $\xi(G;X,Y,Z)$
of \cite{ar:AverbouchGodlinMakowsky10} are generalized to hypergraphs.
$T(H;X,Y)$ is a substitution instance of $\xi(H;X,Y,Z)$ both on graphs and hypergraphs.
In \cite{ar:Kalman-2013} another version of a Tutte polynomial for hypergraphs is proposed.

\begin{problem}
Is $T(H;X,Y)$ almost complete for $r$-uniform hypergraphs.
\end{problem}
Note that the original BPR-conjecture asserts this for graphs, and is still open.
\begin{problem}
Is $\xi(H;X,Y,Z)$ almost complete for $r$-uniform hypergraphs.
\end{problem}

\subsection*{Acknowledgements}

This paper was conceived and written during
the {\em Tutte Centenary retreat 2017}, held at the MATRIX, 
(https://www.matrix-inst.org.au/events/tutte-centenary-retreat/)
Creswick Campus of the University of Melbourne November 25- December 1, 2017.
We would like to thank the organizers  
for inviting both authors and
for the pleasent working atmosphere at the retreat.
We would like to thank Graham Farr for his encouragement and 
Vsevolod Rakita for insightful discussions.

\bibliographystyle{alpha}
\bibliography{Zhang-ref}
\end{document}